\documentclass[12pt]{amsart}

\usepackage{amssymb}
\usepackage{amsmath}
\usepackage{url}
\input xy
\xyoption{all}

\newtheorem{thm}{Theorem}

\newtheorem{Definition}[thm]{Definition}

\newtheorem{Example}[thm]{Example}

\newcommand{\Z}{\mathbb{Z}}

\newcommand{\ba}{\begin{eqnarray*}}
\newcommand{\ea}{\end{eqnarray*}}
\newcommand{\be}{\begin{enumerate}}
\newcommand{\bes}{\begin{enumerate}\topsep=1.5mm \itemsep=-1mm}
\newcommand{\ee}{\end{enumerate}}

\title{A Note on the Critical Group of a Line Graph}
\author{David Perkinson}
\email{davidp@reed.edu}
\author{Nick Salter}
\email{saltern@reed.edu}
\author{Tianyuan Xu}
\email{xut@reed.edu}
\address{Reed College, Portland OR, 97202}

\begin{document} 
\maketitle

\begin{abstract}
  This note answers a question posed by Levine in \cite{levine}.  The main result is
  Theorem~\ref{main theorem} which shows that under certain circumstances a
  critical group of a directed graph is the quotient of a critical group of its
  directed line graph.  
\end{abstract}
\section{Introduction}
Let $G$ be a finite multidigraph with vertices $V$ and edges $E$.  Loops
are allowed in $G$, and we make no connectivity assumptions.  Each edge $e\in E$ has a tail $e^-$ and a target $e^+$. Let $\Z V$
and $\Z E$ be the free abelian groups on $V$ and $E$, respectively.  The
{\em Laplacian}\footnote{The mapping $\Lambda\colon\Z^V\to\Z^V$ defined by
$\Lambda(f)(v)=\sum_{(v,u)\in E}(f(v)-f(u))$ for $v\in V$ is often called the Laplacian of
$G$.  It is the negative $\Z$-dual (i.e., the transpose) of $\Delta_G$.} of $G$ is the $\Z$-linear mapping
$\Delta_G:\Z V\to \Z V$ determined by $\Delta_G(v)=\sum_{(v,u)\in E}(u-v)$ for $v\in
V$.  
Given
$w_*\in V$, define 
\begin{align*}
\phi=\phi_{G,w_*}\colon \Z V&\to\Z V\\
v&\mapsto
\left\{
\begin{array}{cl}
  \Delta_G(v)&\mbox{if $v\neq w_*$},\\
  w_*&\mbox{if $v=w_*$}.
\end{array}
\right.
\end{align*}
The {\em critical group} for $G$ with respect to $w_*$ is the cokernel of
$\phi$:
\[
K(G,w_*):=\mathrm{cok}\,\phi.
\]
The {\em line graph}, $\mathcal{L}G$, for $G$ is the multidigraph whose vertices are the edges of
$G$ and whose edges are $(e,f)$ with $e^+=f^-$.  As with $G$, we have the
Laplacian $\Delta_{\mathcal{L}G}$ and the critical group
$K(\mathcal{L}G,e_*):=\mbox{cok}\,\phi_{\mathcal{L}G,e_*}$ for each $e_*\in E$. 

If every vertex of $G$ has a directed path to $w_*$ then $K(G,w_*)$ is called
the {\em sandpile group} for $G$ with sink $w_*$.  A {\em directed spanning
tree} of $G$ rooted at $w_*$ is a directed subgraph containing all of the
vertices of~$G$, having no directed cycles, and for which $w_*$ has out-degree
$0$ and every other vertex has out-degree $1$.  Let $\kappa(G,w_*)$ denote the
number of directed spanning trees rooted at $w_*$. It is a well-known
consequence of the matrix-tree theorem that the number of elements of the
sandpile group with sink $w_*$ is equal to $\kappa(G,w_*)$.  For a basic
exposition of the properties of the sandpile group, the reader is referred to
\cite{HLPW}.

In his paper, \cite{levine}, Levine shows that if $e_*=(w_*,v_*)$, then
$\kappa(G,w_*)$ divides $\kappa(\mathcal{L}G,e_*)$ under the hypotheses of our
Theorem~\ref{main theorem}.  This leads him to ask the natural question as to
whether $K(G,w_*)$ is a subgroup or quotient of $K(\mathcal{L}G,e_*)$.  In this
note, we answer this question affirmatively by demonstrating a surjection
$K(\mathcal{L}G,e_*)\to K(G,w_*)$.  Further, in the case in which the out-degree
of each vertex of $G$ is a fixed integer $k$, we show the kernel of this
surjection is the $k$-torsion subgroup of $K(\mathcal{L}G,e_*)$.  These results
appear as Theorem~\ref{main theorem} and may be seen as analogous to Theorem~1.2
of \cite{levine}.  In \cite{levine}, partially for convenience, some assumptions
are made about the connectivity of $G$ which are not made in this note.

For related work on the critical group of a line graph for an undirected graph, see \cite{reiner}.
\section{Results}
Fix $e_*=(w_*,v_*)\in E$.  Define the modified target mapping
\begin{align*}
\tau\colon \Z E&\to\Z V\\
e&\mapsto
\left\{
\begin{array}{cl}
  e^+&\mbox{if $e\neq e_*$},\\
  0&\mbox{if $e=e_*$}.
\end{array}
\right.
\end{align*}
Also define 
\begin{align*}
\rho\colon \Z E&\to\Z V\\
e&\mapsto
\left\{
\begin{array}{cl}
  \Delta_G(w_*)-v_*-w_*+e^+&\mbox{if $e\neq e_*$},\\
  0&\mbox{if $e=e_*$}.
\end{array}
\right.
\end{align*}

Let $k$ be a positive integer.  The graph $G$ is {\em $k$-out-regular} if the
out-degree of each of its vertices is $k$.
\begin{thm}\label{main theorem}
  If $\mathrm{indeg}(v)\geq 1$ for all $v\in V$ and $\mathrm{indeg}(v_*)\geq 2$,
  then 
  \[
  \rho\colon \Z E\to \Z V
  \]
  descends to a surjective homomorphism $\overline{\rho}\colon K(\mathcal{L}G,e_*)\to K(G,w_*)$.
  
  Moreover, if $G$ is $k$-out-regular, the kernel of $\overline{\rho}$ is the $k$-torsion
  subgroup of $K(\mathcal{L}G,e_*)$.
\end{thm}
\begin{proof}
  Let $\rho_0\colon\Z V\to \Z V$ be the homomorphism defined on vertices $v\in
  V$ by
  \[
  \rho_0(v):=\Delta_{G}(w_*)-v_*-w_*+v
  \]
  so that $\rho=\rho_0\circ\tau$.  The mapping $\rho_0$ is an isomorphism, its
  inverse being itself:
  \begin{align*}
    \rho_0^2(v)&=\rho_0(\Delta_{G}(w_*)-v_*-w_*+v)\\
    &=\sum_{e^-=w_*}(\rho_0(e^+)-\rho_0(w_*))-\rho_0(v_*)-\rho_0(w_*)+\rho_0(v)\\
    &=\Delta_{G}(w_*)-\rho_0(v_*)-\rho_0(w_*)+\rho_0(v)\\
    &=v.
  \end{align*}

  Let $\psi\colon\Z V\to\Z V$ be the homomorphism defined on vertices 
  $v\in V$ by
  \[
  \psi(v):=
  \begin{cases}
    \Delta_{G}(v)&\text{if $v\neq w_*$},\\
    \Delta_{G}(w_*)-v_*&\text{if $v=w_*$}.
  \end{cases}
  \]
  Let $\phi_{G}$ and $\phi_{\mathcal{L}G}$ denote $\phi_{G,w_*}$ and
  $\phi_{\mathcal{L}G,e_*}$, respectively.
We claim the following diagram commutes:
\[
\xymatrix{
\Z E\ar[d]_{\tau}\ar[rr]^{\phi_{\mathcal{L}G}}&
&\Z E \ar[d]^{\tau}\\
\Z V\ar@{=}[d]\ar[rr]^{\psi}&
&\Z V\ar[d]^{\rho_0}\\
\Z V\ar[rr]^{\phi_G}&&\Z V.
}
\]
To prove commutativity of the top square of the diagram, first suppose $e\neq e_*$.
Then
\[
\tau(\phi_{\mathcal{L}G}(e))=\tau(\Delta_{\mathcal{L}G}(e))=\tau\left(\sum_{f^-=e^+}(f-e)\right).
\]
If $e\neq e_*$ and $e^+\neq w_*$, then
\[
\tau\left(\sum_{f^-=e^+}(f-e)\right)=\sum_{f^-=e^+}(f^+-e^+)=\Delta_{G}(e^+)=\psi(\tau(e)).
\]
On the other hand, if $e\neq e_*$ and $e^+=w_*$, then 
\begin{align*}
\tau\left(\sum_{f^-=e^+}(f-e)\right)&=\sum_{f^-=e^+,f\neq
e_*}(f^+-e^+)+\tau(e_*-e)\\
&=\sum_{f^-=e^+,f\neq e_*}(f^+-e^+)-w_*\\
&=\Delta_G(w_*)-v_*=\psi(\tau(e)).
\end{align*}
Therefore, $\tau(\phi_{\mathcal{L}G}(e))=\psi(\tau(e))$ holds if $e\neq e_*$.
Moreover, the equality still holds if $e=e_*$ since $\tau(e_*)=0$.  Hence, the
top square of the diagram commutes. 

To prove that the bottom square of the diagram commutes, there are two cases.
First, if $v\neq w_*$, then 
\[
\rho_0(\psi(v))=\sum_{(v,u)\in E}(\rho_0(u)-\rho_0(v))=\sum_{(v,u)\in
E}(u-v)=\Delta_G(v)=\phi_G(v).
\]
Second, if $v=w_*$, then
\[
\rho_0(\psi(v))=\rho_0(\Delta_G(w_*)-v_*)=\Delta_G(w_*)-\rho_0(v_*)=w_*=\phi_G(v).
\]

From the commutativity of the diagram, the cokernel of $\psi$ is isomorphic to $K(G,w_*)$, and
$\rho=\rho_0\circ\tau$ descends to a homomorphism $\overline{\rho}\colon
K(\mathcal{L}G,e_*)\to K(G,w_*)$ as claimed.  The hypothesis on the in-degrees of
the vertices assures that~$\tau$, hence~$\overline{\rho}$, is surjective.

Now suppose that $G$, hence $\mathcal{L}$G, is $k$-out-regular.  This part of
our proof is an adaptation of that given for Theorem 1.2 in \cite{levine}. Since $\rho_0$
is an isomorphism, it suffices to show that the kernel of the induced map,
$\overline{\tau}\colon K(\mathcal{L}G,e_*)\to\mathrm{cok}\,\psi$ has kernel
equal to the $k$-torsion of $K(\mathcal{L}G,e_*)$.  To this end, define the
homomorphism $\sigma\colon\Z V\to\Z E$, given on vertices $v\in V$ by 
\[
\sigma(v):=\sum_{e^-=v}e.
\]
We claim that the image of $\sigma\circ\psi$ lies in the image of
$\phi_{\mathcal{L}G}$, so that $\sigma$ induces a map, $\overline{\sigma}$, between $\mathrm{cok}\,\psi$ and
$K(\mathcal{L}G,e_*)$.  To see this, first note that for $v\in V$,
\begin{align*}
  \sigma(\Delta_G(v))&=\sigma\left(\sum_{e^-=v}e^+-kv\right)\\
  &=\sum_{e^-=v}\sum_{f^-=e^+}f-k\sum_{e^-=v}e\\
  &=\sum_{e^-=v}\Delta_{\mathcal{L}G}(e)
\end{align*}
Therefore, for $v\neq w_*$, it follows that $\sigma(\psi(v))$ is in the image
of $\phi_{\mathcal{L}G}$.  On the other hand, using the calculation just made,
\begin{align*}
  \sigma(\Delta_G(w_*)-v_*)&=\sum_{e^-=w_*}\Delta_{\mathcal{L}G}(e)-\sum_{f^-=v^*}f\\
  &=\sum_{e^-=w_*}\Delta_{\mathcal{L}G}(e)-\left(\sum_{f^-=v^*}f-k\,e_*+k\,e_*\right)\\
  &=\sum_{e^-=w_*}\Delta_{\mathcal{L}G}(e)-\Delta_{\mathcal{L}G}(e_*)-k\,e_*\\
  &=\sum_{e^-=w_*,e\neq e_*}\Delta_{\mathcal{L}G}(e)-k\,e_*,
\end{align*}
which is also in the image of $\phi_{\mathcal{L}G}$.

Now note that for $e\neq e_*$, 
\[
\overline{\sigma}(\overline{\tau}(e))=\sum_{f^-=e^+}f=\Delta_{\mathcal{L}G}(e)+k\,e
=k\,e\in K(\mathcal{L}G,e_*).
\]
Thus, the kernel of $\overline{\tau}$ is contained in the $k$-torsion of
$K(\mathcal{L}G,e_*)$, and to show equality it suffices to show that
$\overline{\sigma}$ is injective.

The case where $k=1$ is trivial since there are no $G$ satisfying the
hypotheses:  if $G$ is $1$-out-regular and $\mathrm{indeg}(v)\geq1$ for
all $v\in V$, then $\mathrm{indeg}(v)=1$ for all $v\in V$, including $v_*$.
So suppose that $k>1$ and that $\eta=\sum_{v\in V}a_v\,v$ is in the kernel
of~$\overline{\sigma}$.  We then have
\begin{equation}\label{eqn1}
\sigma(\eta)=\sum_{v\in V}\sum_{e^-=v}a_v\,e=\sum_{e\neq
e_*}b_e\,\Delta_{\mathcal{L}G}(e)+c\,e_*
\end{equation}
for some integers $b_e$ and $c$.  Comparing coefficients in (\ref{eqn1}) gives
\begin{equation}
  a_{e^-}=\sum_{f^+=e^-,f\neq e_*}b_f-k\,b_{e}\qquad\text{for $e\neq e_*$}.
  \label{eqn2}
\end{equation}
  Define
\[
F(v)=\frac{1}{k}\left(\sum_{f^+=v, f\neq e_*}b_f-a_v\right).
\]
From (\ref{eqn2}),
\begin{equation}
F(e^-)=b_e\qquad\text{for $e\neq e_*$}.
\label{eqn3}
\end{equation}

Since $k>1$, for each vertex $v$, we can choose an
edge $e_v\neq e_*$ with $e_v^-=v$.
By (\ref{eqn2}) and (\ref{eqn3}), for all $v\in V$, 
\[
a_v=\sum_{f^+=v,f\neq e_*}b_f-k\,b_{e_v}=\sum_{f^+=v,f\neq e_*}F(f^-)-k\,F(v).
\]
Therefore, as an element of $\mbox{cok}\,\psi$,
\begin{align*}
\eta &= \sum a_v v = \sum_{e\neq e_*} F\left(e^-\right)e^+ - \sum_{v\in V}kF(v)v\\
&= \sum_{v\in V, v\neq w_*}F(v)\left(\sum_{e^-=v} e^+- kv\right) + F(w_*)\left(\sum_{e^-=w_*,
e\neq e_*} e^+ - kw_*\right)\\
&=  \sum_{v\in V, v\neq w_*}F(v)\Delta_{G}(v)+ F(w_*)(\Delta_{G}(w_*)-v_*)\\
&= 0,
\end{align*}
which shows that $\overline{\sigma}$ is injective.
\end{proof}

\bibliography{mybib}{}
\bibliographystyle{plain}

\end{document}